\documentclass[reqno,12pt]{amsart}
\usepackage[english]{babel}
\usepackage{amsmath,amsfonts,amsthm,amssymb,amsbsy,upref,color,graphicx,hyperref,enumerate,comment}
\usepackage[a4paper,margin=3.06truecm]{geometry}

\DeclareMathOperator{\dive}{div}

\def\ds{\displaystyle}
\def\eps{{\varepsilon}}
\def\N{\mathbb{N}}
\def\O{\Omega}
\def\R{\mathbb{R}}
\def\A{\mathcal{A}}
\def\HH{\mathcal{H}}
\newcommand{\be}{\begin{equation}}
\newcommand{\ee}{\end{equation}}
\newcommand{\bib}[4]{\bibitem{#1}{\sc#2: }{\it#3. }{#4.}}

\def\avint{\mathop{\rlap{\hskip2.5pt---}\int}\nolimits}

\numberwithin{equation}{section}
\theoremstyle{plain}

\newtheorem{theo}{Theorem}[section]

\newtheorem{prop}[theo]{Proposition}

\newtheorem{conj}[theo]{Conjecture}
\theoremstyle{remark}
\newtheorem{rema}[theo]{Remark}

\title[Some inequalities involving perimeter and torsional rigidity]{Some inequalities involving perimeter and torsional rigidity}

\author[L. Briani]{Luca Briani}

\author[G. Buttazzo]{Giuseppe Buttazzo}

\author[F. Prinari]{Francesca Prinari}

\date{}

\begin{document}

\begin{abstract}
We consider shape functionals of the form $F_q(\O)=P(\O)T^q(\O)$ on the class of open sets of prescribed Lebesgue measure. Here $q>0$ is fixed, $P(\O)$ denotes the perimeter of $\O$ and $T(\O)$ is the torsional rigidity of $\O$. The minimization and maximization of $F_q(\O)$ is considered on various classes of admissible domains $\O$: in the class $\A_{all}$ of {\it all domains}, in the class $\A_{convex}$ of {\it convex domains}, and in the class $\A_{thin}$ of {\it thin domains}.

\end{abstract}

\maketitle



%%%%%%%%%%%%%%%%%%%%%%%%%%%%%%%%%%%%%%%%%%%%%%%%%%
\section{Introduction\label{sintro}}

In this paper, given an open set $\O\subset\R^d$ with finite Lebesgue measure, we consider the quantities
\[\begin{split}
&P(\O)=\text{perimeter of }\O;\\
&T(\O)=\text{torsional rigidity of }\O.
\end{split}\]
The perimeter $P(\O)$ is defined according to the De Giorgi formula
$$P(\O)=\sup\left\{\int_\O\dive\phi\,dx\ :\ \phi\in C^1_c(\R^d;\R^d),\ \|\phi\|_{L^\infty(\R^d)}\le1\right\}.$$
The {\it scaling property} of the perimeter is
$$P(t\O)=t^{d-1}P(\O)\qquad\text{for every }t>0$$
and the relation between $P(\O)$ and the Lebesgue measure $|\O|$ is the well-known {\it isoperimetric inequality}:
\be\label{isoper}
\frac{P(\O)}{|\O|^{(d-1)/d}}\ge\frac{P(B)}{|B|^{(d-1)/d}}
\ee
where $B$ is any ball in $\R^d$. In addition, the inequality above becomes an equality if and only if $\O$ is a ball (up to sets of Lebesgue measure zero).\\
The torsional rigidity $T(\O)$ is defined as
$$T(\O)=\int_\O u\,dx$$
where $u$ is the unique solution of the PDE
\be\label{pdetorsion}\begin{cases}
-\Delta u=1&\text{in }\O,\\
u\in H^1_0(\O).
\end{cases}
\ee
Equivalently, $T(\O)$ can be characterized through the maximization problem
$$T(\O)=\max\Big\{\Big[\int_\O u\,dx\Big]^2\Big[\int_\O|\nabla u|^2\,dx\Big]^{-1}\ :\ u\in H^1_0(\O)\setminus\{0\}\Big\}.$$
Moreover $T$ is increasing with respect to the set inclusion, that is
$$\O_1\subset\O_2\Longrightarrow T(\O_1)\le T(\O_2)$$
and $T$ is additive on disjoint families of open sets. The scaling property of the torsional rigidity is
$$T(t\O)=t^{d+2}T(\O),\qquad\text{for every }t>0,$$
and the relation between $T(\O)$ and the Lebesgue measure $|\O|$ is the well-known {\it 
Saint-Venant inequality} (see for instance \cite{he06}, \cite{hepi05}):
\be\label{stven}
\frac{T(\O)}{|\O|^{(d+2)/d}}\le\frac{T(B)}{|B|^{(d+2)/d}}.
\ee
Again, the inequality above becomes an equality if and only if $\O$ is a ball (up to sets of capacity zero). If we denote by $B_1$ the unitary ball of $\R^d$ and by $\omega_d$ its Lebesgue measure, then the solution of \eqref{pdetorsion}, with $\O=B_1$, is
$$u(x)=\frac{1-|x|^2}{2d}$$
which provides
\be\label{torball}
T(B_1)=\frac{\omega_d}{d(d+2)}.
\ee

We are interested in the problem of minimizing or maximizing quantities of the form
$$P^\alpha(\O)T^\beta(\O)$$
on some given class of open sets $\O\subset\R^d$ having a prescribed Lebesgue measure $|\O|$, where $\alpha,\beta$ are two given exponents. Similar problems have been considered for shape functionals involving:
\begin{itemize}
\item[-]the torsional rigidity and the first eigenvalue of the Laplacian in \cite{bbp20}, \cite{bfnt16}, \cite{bra14}, \cite{bgm17}, \cite{dpga14}, \cite{kj78a}, \cite{kj78b}, \cite{luzu19};
\item[-]the torsional rigidity and the Newtonian capacity in \cite{bb20};
\item[-]the perimeter and the first eigenvalue of the Laplacian in \cite{flxx};
\item[-]the perimeter and the Newtonian capacity in \cite{CFG05}, \cite{IGP11}.
\end{itemize}

The case $\beta=0$ reduces to the isoperimetric inequality, and we have, denoting by $\O^*_m$ a ball of measure $m$,
$$\begin{cases}
\min\big\{P(\O)\ :\ |\O|=m\big\}=P(\O^*_m)\\
\sup\big\{P(\O)\ :\ |\O|=m\big\}=+\infty.
\end{cases}$$
Similarly, in the case $\alpha=0$, the Saint Venant inequality yields 
$$\max\big\{T(\O)\ :\ |\O|=m\big\}=T(\O^*_m)=\frac{m}{d(d+2)}\Big(\frac{m}{\omega_d}\Big)^{2/d}$$
while 
$$\inf\big\{T(\O)\ :\ |\O|=m\big\}=0.$$
Indeed if we choose $\O_n=\cup_{k=1}^n B_{n,k}$ where $B_{n,k}$ are disjoint balls of measure $m/n$ each, we get for every $n\in\N$
$$\inf\big\{T(\O)\ :\ |\O|=m\big\}\le T(\Omega_n)=\frac{m^{(d+2)/d}}{d(d+2)\omega_d^{2/d}}\,n^{-2/d}.$$

The case when $\alpha$ and $\beta$ have a different sign is also immediate; for instance, if $\alpha>0$ and $\beta<0$ we have from \eqref{isoper} and \eqref{stven}
$$\begin{cases}
\min\big\{P^\alpha(\O)T^\beta(\O)\ :\ |\O|=m\big\}=P^\alpha(\O^*_m)T^\beta(\O^*_m)\\
\sup\big\{P^\alpha(\O)T^\beta(\O)\ :\ |\O|=m\big\}=+\infty,
\end{cases}$$
and similarly, if $\alpha<0$ and $\beta>0$ we have
$$\begin{cases}
\inf\big\{P^\alpha(\O)T^\beta(\O)\ :\ |\O|=m\big\}=0\\
\max\big\{P^\alpha(\O)T^\beta(\O)\ :\ |\O|=m\big\}=P^\alpha(\O^*_m)T^\beta(\O^*_m).
\end{cases}$$

The cases we will investigate are the remaining ones; with no loss of generality we may assume $\alpha=1$, so that the optimization problems we consider are for the quantities
$$P(\O)T^q(\O),\qquad\text{with }q>0.$$
In order to remove the Lebesgue measure constraint $|\O|=m$ we consider the {\it scaling free} functionals
$$F_q(\O)=\frac{P(\O)T^q(\O)}{|\O|^{\alpha_q}}\qquad\text{with }\alpha_q=1+q+\frac{2q-1}{d}.$$
In the following sections we study the minimization and the maximization problems for the shape functionals $F_q$ on various classes of domains. More precisely we consider the cases below.

The class of {\it all} domains $\O$ (nonempty)
$$\A_{all}=\big\{\O\subset\R^d\ :\ \O\ne\emptyset\big\}$$
will be considered in Section \ref{sall}; we show that for every $q>0$ both the maximization and the minimization problems for $F_q$ on $\A_{all}$ are ill posed.

The class of {\it convex} domains $\O$
$$\A_{convex}=\big\{\O\subset\R^d\ :\ \O\ne\emptyset,\ \O\text{ convex}\big\}$$
will be considered in Section \ref{sconvex}; we show that for $0<q<1/2$ the maximization problem for $F_q$ on $\A_{convex}$ is ill posed, whereas the minimization problem is well posed. On the contrary, when $q>1/2$ the minimization problem for $F_q$ on $\A_{convex}$ is ill posed, whereas the maximization problem is well posed. In the threshold case $q=1/2$ the precise value of the infimum of $F_{1/2}$ is provided; concerning the precise value of the supremum of $F_{1/2}$ an interesting conjecture is stated. At present, the conjecture has been shown to be true in the case $d=2$, while the question is open in higher dimensions.

The class of thin domains $\A_{thin}$, suitably defined, will be considered in Section \ref{sthin}. If $h(s)$ represents the asymptotical {\it local thickness} of the thin domain as $s$ varies in a $d-1$ dimensional domain $A$, the maximization of the functional $F_{1/2}$ on $\A_{thin}$ reduces to the maximization of a functional defined on nonnegative functions $h$ defined on $A$; this allows us to prove the conjecture for any dimension $d$ on the class of {\it thin convex} domains.

%\begin{itemize}
%\item[-]the class of {\it all} domains $\O$ (nonempty)
%$$\A_{all}=\big\{\O\subset\R^d\ :\ \O\ne\emptyset\big\};$$
%\item[-]the class of {\it convex} domains $\O$
%$$\A_{convex}=\big\{\O\subset\R^d\ :\ \O\ne\emptyset,\ \O\text{ convex}\big\};$$
%\item[-]the class of {\it thin} domains $\A_{thin}$, suitably defined in Section \ref{sthin}.
%\end{itemize}

%%%%%%%%%%%%%%%%%%%%%%%%%%%%%%%%%%%%%%%%%%%%%%%%%%
\section{Optimization in the class of all domains\label{sall}}

In this section we show that the minimization and the maximization problems for the shape functionals $F_q$ are both ill posed, for every $q>0$.

\begin{theo}\label{tall}
There exist two sequences $\O_{1,n}$ and $\O_{2,n}$ of smooth domains such that for every $q>0$ we have
$$F_q(\O_{1,n})\to0\qquad\text{and}\qquad F_q(\O_{2,n})\to+\infty.$$
In particular, we have
$$\begin{cases}
\inf\big\{F_q(\O)\ :\ \O\in\A_{all},\ \O\text{ smooth}\big\}=0\\
\sup\big\{F_q(\O)\ :\ \O\in\A_{all},\ \O\text{ smooth}\big\}=+\infty.
\end{cases}$$
\end{theo}

\begin{proof}
In order to show the $\sup$ equality it is enough to take as $\O_{2,n}$ a perturbation of the unit ball $B_1$ such that
$$B_{1/2}\subset\O_{2,n}\subset B_2\qquad\text{and}\qquad P(\O_{2,n})\to+\infty.$$
Then we have
$$|\O_{2,n}|\le|B_2|,\qquad T(\O_{2,n})\ge T(B_{1/2}),$$
where we used the monotonicity of the torsional rigidity. Then
$$F_q(\O_{2,n})\ge\frac{P(\O_{2,n})T^q(B_{1/2})}{|B_2|^{\alpha_q}}\to+\infty.$$
In order to prove the $\inf$ equality we take as $\O_\eps$ the unit ball $B_1$ to which we remove a periodic array of holes; the centers of two adjacent holes are at distance $\eps$ and the radii of the holes are
$$r_\eps=\begin{cases}
e^{-1/(c\eps^2)}&\text{if }d=2\\
c\eps^{d/(d-2)}&\text{if }d>2.
\end{cases}$$
It is easy to see that, as $\eps\to0$, we have
$$|\O_\eps|\to|B_1|\qquad\text{and}\qquad P(\O_\eps)\to P(B_1).$$
Concerning the torsion $T(\O_\eps)$, we have (see \cite{ciomur})
$$T(\O_\eps)\to\int_{B_1} u_c\,dx$$
where $u_c$ is the nonnegative function which solves 
$$\begin{cases}
-\Delta u_c+K_cu_c=1&\text{in }B_1\\
u_c\in H^1_0(B_1),
\end{cases}$$
being $K_c$ the constant
$$K_c=\begin{cases}
c\pi/2&\text{if }d=2\\
d(d-2)2^{-d}\omega_d c^{d-2}&\text{if }d>2.
\end{cases}$$
Since for every $c>0$ we have that
$$\int_{B_1}|\nabla u_c(x)|^2+K_c u^2_c(x)\,dx=\int_{B_1} u_c\,dx$$
we get that
$$\int_{B_1}u_c\,dx\le\frac{\omega_d} {K_c}.$$
Therefore, a diagonal argument allows us to construct a sequence $\O_{1,n}$ such that
$$|\O_{1,n}|\to|B_1|,\qquad P(\O_{1,n})\to P(B_1),\qquad T(\O_{1,n})\to0,$$
which concludes the proof.
\end{proof}

%%%%%%%%%%%%%%%%%%%%%%%%%%%%%%%%%%%%%%%%%%%%%%%%%%
\section{Optimization in the class of convex domains\label{sconvex}}

In this section we consider only domains $\O$ which are {\it convex}. A first remark is in the proposition below and shows that in some cases the optimization problems for the shape functional $F_q$ is still ill posed.

\begin{prop}\label{illq}
We have
\[\begin{cases}
\inf\big\{F_q(\O)\ :\ \O\in\A_{convex}\big\}=0&\text{for every }q>1/2;\\
\sup\big\{F_q(\O)\ :\ \O\in\A_{convex}\big\}=+\infty&\text{for every }q<1/2.
\end{cases}\]
\end{prop}

\begin{proof}
Let $A$ be a smooth convex $d-1$ dimensional set and for every $\eps>0$ consider the domain $\O_\eps\in\A_{convex}$ given by
$$\O_\eps=A\times]-\eps/2,\eps/2[.$$
We have (for the torsion asymptotics see for instance \cite{bbp20})
\[\begin{split}
&P(\O_\eps)\approx2\HH^{d-1}(A),\\
&T(\O_\eps)\approx\frac{\eps^3}{12}\HH^{d-1}(A),\\
&|\O_\eps|=\eps\HH^{d-1}(A),
\end{split}\]
so that
\be\label{slab}F_q(\O_\eps)\approx\frac{2}{12^q\big(\HH^{d-1}(A)\big)^{(2q-1)/d}}\,\eps^{(2q-1)(d-1)/d}.\ee
Letting $\eps\to0$ achieves the proof.
\end{proof}

We show now that in some other cases the optimization problems for the shape functional $F_q$ is well posed. Let us begin to consider the case $q=1/2$.

\begin{prop}\label{polya}
We have
\be\label{polyaineq}
\inf\big\{F_{1/2}(\O)\ :\ \O\in\A_{convex}\big\}=3^{-1/2}
\ee
and the infimum is asymptotically reached by domains of the form
$$\O_\eps=A\times]-\eps/2,\eps/2[$$
as $\eps\to0$, where $A$ is any $d-1$ dimensional convex set.
\end{prop}

\begin{proof}
Thanks to a classical result by Polya (\cite{polya60}, see also Theorem 5.1 of \cite{dpga14}) it holds
$$T(\O)\ge\frac13\frac{|\O|^{3}}{(P(\O))^2}.$$
Then
$$F_{1/2}(\O)=\frac{P(\O)(T(\O))^{1/2}}{|\O|^{3/2}}\ge3^{-1/2}$$
for any bounded open convex set. Taking into account \eqref{slab}, we get \eqref{polyaineq}.
\end{proof}

Concerning the supremum of $F_{1/2}(\O)$ in the class $\A_{convex}$ we can only show that it is finite.

\begin{prop}\label{finite}
For every $\O\in\A_{convex}$ we have
\be\label{finitebound} F_{1/2}(\O)\le\frac{2^d d^{3d/2}}{\omega_d}\sqrt{\frac{d}{d+2}}\;\ee.
\end{prop}

\begin{proof}
By the John's ellipsoid Theorem \cite{J48}, there exists an ellipsoid, that without loss of generality we may assume centered at the origin,
$$E_a=\bigg\{x\in\R^d\ :\ \sum_{i=1}^d\frac{x_i^2}{a_i^2}<1\bigg\},\qquad a=(a_1,\dots,a_d),\text{ with }a_i>0$$
such that $E_a\subset\O\subset dE_a$. Then we have
\be\label{ineqf12}
F_{1/2}(\O)\le\frac{P(dE_a)\big(T(dE_a)\big)^{1/2}}{|E_a|^{3/2}}.
\ee
Since the solution of \eqref{pdetorsion} for $E_a$ is given by
$$u(x)=\frac12\bigg(\sum_{i=1}^d a_i^{-2}\bigg)^{-1}\bigg(1-\sum_{i=1}^{d}\frac{x_i^2}{a_i^2}\bigg),$$
we obtain
$$T(E_a)=\frac{\omega_d}{d+2}\bigg(\sum_{i=1}^d a_i^{-2}\bigg)^{-1}\prod_{i=1}^d a_i,$$
while
$$|E_a|=\omega_d\prod_{i=1}^d a_i.$$
To estimate $P(E_a)$ we notice that $E_a$ is contained in the cuboid $Q=\prod_1^d]-a_i,a_i[$, so that
$$P(E_a)\le P(Q)=2\sum_{i=1}^d\prod_{j\ne i}(2a_j)=2^d\bigg(\sum_{i=1}^d\frac{1}{a_i}\bigg)\prod_{i=1}^d a_i.$$ 
Combining these formulas we have from \eqref{ineqf12}
$$F_{1/2}(\O)\le\frac{2^d d^{3d/2}}{\omega_d(d+2)^{1/2}}\bigg(\sum_{i=1}^d\frac{1}{a_i}\bigg)\bigg(\sum_{i=1}^d\frac{1}{a_i^2}\bigg)^{-1/2}$$
and finally, by Jensen inequality,
$$F_{1/2}(\O)\le\frac{2^d d^{3d/2}}{\omega_d}\sqrt{\frac{d}{d+2}}\;,$$
as required.
\end{proof}

On the precise value of $\sup\big\{F_{1/2}(\O):\O\in\A_{convex}\big\}$ we make the following conjecture.

\begin{conj}\label{conjecture}
We have
$$\sup\big\{F_{1/2}(\O)\ :\ \O\in\A_{convex}\big\}=d\Big(\frac{2}{(d+1)(d+2)}\Big)^{1/2}$$
and it is asymptotically reached by taking for instance
$$\O_\eps=\big\{(s,t)\ :\ s\in A,\ 0<t<\eps(1-|s|)\big\}$$
as $\eps\to0$, where $A$ is the unit ball in $\R^{d-1}$.
\end{conj}

\begin{rema}
We recall that Conjecture \ref{conjecture} has been shown to be true in the case $d=2$ (see \cite{polya60}, \cite{makai62}, and the more recent paper \cite{FGL13}). In Section \ref{sthin} we prove the conjecture above for every $d\ge2$ in the class of convex thin domains.
\end{rema}

We show now that for $F_q$ in the class $\A_{convex}$ the minimization problem is well posed when $q<1/2$ and the maximization problem is well posed when $q>1/2$. From the bounds obtained in Propositions \ref{polya} and \ref{finite} we can prove the following results.

\begin{prop}
We have
\[\begin{cases}
\inf\big\{F_q(\O)\ :\ \O\in\A_{convex}\big\}\ge3^{-1/2}\big(d(d+2)\big)^{1/2-q}\omega_d^{(1-2q)/d}&\text{for every }q\le1/2\\
\sup\big\{F_q(\O)\ :\ \O\in\A_{convex}\big\}\le\ds\frac{2^d d^{3d/2-q+1}}{(d+2)^q\omega_d^{1+(2q-1)/d}}&\text{for every }q\ge1/2.
\end{cases}\]
\end{prop}

\begin{proof}
We have
$$F_q(\O)=F_{1/2}(\O)\left(\frac{T(\O)}{|\O|^{(d+2)/d}}\right)^{q-1/2}.$$
Hence it is enough to apply the bounds \eqref{polyaineq} and \eqref{finitebound}, together with the Saint Venant inequality \eqref{stven} to get that for every $\O\in\A_{convex}$
\[\begin{split}
&\inf\big\{F_q(\O)\ :\ \O\in\A_{convex}\big\}\ge3^{-1/2}\left(\frac{T(B)}{B^{(d+2)/d}}\right)^{q-1/2}\qquad\text{if }q\le1/2\\
&\sup\big\{F_q(\O)\ :\ \O\in\A_{convex}\big\}< \frac{2^d d^{3d/2}}{\omega_d}\sqrt{\frac{d}{d+2}}\left(\frac{T(B)}{B^{(d+2)/d}}\right)^{q-1/2}\qquad\text{if }q\ge1/2.
\end{split}\]
By the expression \eqref{torball} for $T(B)$ we conclude the proof.
\end{proof}

We now prove the existence of a convex minimizer when $q<1/2$ and of a convex maximizer when $q>1/2$.

\begin{theo}
There exists a solution for the following optimization problems:
\[\begin{cases}
\min\big\{F_q(\O)\ :\ \O\in\A_{convex}\big\}&\text{for every }q<1/2;\\
\max\big\{F_q(\O)\ :\ \O\in\A_{convex}\big\}&\text{for every }q>1/2.
\end{cases}\]
\end{theo}

\begin{proof}
Suppose $q<1/2$ and consider $\O_n$ a minimizing sequence for $F_q(\O)$. By the John's ellipsoid Theorem we can assume that there exists a sequence of ellipsoids $E_{a_n}$ such that 
$$E_{a_n}\subset\O_n\subset dE_{a_n}.$$
By rotations, translations and scaling invariance of $F_q$ we can assume without loss of generality that 
$$E_{a_n}=\bigg\{x\in\R^d\ :\ \sum_{i=1}^d\frac{x_{i}^{2}}{a_{in}^{2}}< 1\bigg\},\quad a_n=(a_{1n},\dots,a_{dn}),\ 0<a_{1n}\le\dots\le a_{dn}=1.$$
Observe that this implies that the diameter of $\O_n$ is uniformly bounded in $n$. We claim that
$$a_{1n}\ge c\qquad\text{for every }n\in\N$$
where $c$ is a positive constant. Then the proof is achieved by extracting a subsequence $\O_{n_k}$ which converges both in the sense of characteristic functions and in the Hausdorff metric to some open, non empty, convex, bounded set $\O^-$ and by using the continuity properties of torsional rigidity, perimeter and volume (see for instance, \cite{bubu05}, \cite{hepi05}).

To prove the claim we use a strategy similar to the one already used in the proof of Proposition \ref{finite}. Let $Q_{a_n}$ be the cuboid $\prod_{i=1}^d]-a_{in},a_{in}[$. Since
$$d^{-1/2}Q_{a_n}\subset E_{a_n}$$
we have, for $n$ large enough,
\be\label{contoa0}
F_q(B_1)\ge F_q(\O_n)\ge\frac{1}{d^{(d-1)/2}d^{d\alpha_q}}\frac{T^q(E_{a_n})P(Q_{a_n})}{|E_{a_n}|^{\alpha_q}}.
\ee
An explicit computation shows
$$\frac{T^q(E_{a_n})P(Q_{a_n})}{|E_{a_n}|^{\alpha_q}}
=\frac{2^d\omega_d^{q-\alpha_q}}{(d+2)^q}\left(\frac{\sum_{i=1}^d a_{in}^{-1}}{\big(\sum_{i=1}^d a_{in}^{-2}\big)^{1/2}}\right)\left(\frac{\big(\sum_{i=1}^d a_{in}^{-2}\big)^{1/2}}{(\prod_{i=1}^d a_{in}^{-1})^{1/d}}\right)^{1-2q}.$$
Observe that, by Cauchy-Schwarz inequality,
\be\label{stimamid_a1}
1\le\frac{\sum_{i=1}^d a_{in}^{-1}}{(\sum_{i=1}^d a_{in}^{-2})^{1/2}}\le\sqrt{d},
\ee
while for the last term it holds
\be\label{stima_a1}\frac{\big(\sum_{i=1}^d a_{in}^{-2}\big)^{1/2}}{(\prod_{i=1}^d a_{in}^{-1})^{1/d}}
=\frac{\big(\sum_{i=1}^d a_{in}^{-2}\big)^{1/2}}{(\prod_{i=1}^{d-1} a_{in}^{-1})^{1/d}}
\ge\frac{a_{1n}^{-1}}{\big(a_{1n}^{-1}\big)^{(d-1)/d}}=\left(\frac{1}{a_{1n}}\right)^{1/d}
\ee
%\be\label{stima_a1}
%\frac{(\prod_{i=1}^{d}a_{in})^{1/d}}{(\sum_{i=1}^{d}a_{in}^{-2})^{-1/2}}=\frac{(\sum_{i=1}^{d}\prod_{j\neq i}a^{2}_{in})^{1/2}}{(\prod_{i}^{d}a_{in})^{1-1/d}}\geq \frac{(\prod_{i=2}^{d}a_{in})^{1/d}}{a_{1n}^{(d-1)/d}}\geq \left(\frac{1}{a_{1n}}\right)^{1/d}.
%\ee
Therefore, putting together \eqref{contoa0}--\eqref{stima_a1} and using the fact that $q<1/2$ we obtain that, if $n$ is large enough, the sequence $a_{1n}$ must be greater than some positive constant $c$, which proves the claim.

The case $q>1/2$ can be proved in a similar way. If $\O_n$ is a maximizing sequence for $F_q(\O)$ and $E_{a_n}$ are ellipsoids such that $E_{a_n}\subset\O_n\subset dE_{a_n}$, we have
\be\label{ineqmax}
F_q(B_1)\le F_q(\O_n)\le\frac{P(dE_{a_n})T^q(dE_{a_n})}{|E_{a_n}|^{\alpha_q}}=d^{d-1+q(d+2)}\frac{P(E_{a_n})T^q(E_{a_n})}{|E_{a_n}|^{\alpha_q}}\;.
\ee
If $Q_{a_n}$ is the cuboid $\prod_{i=1}^d]-a_{in},a_{in}[$ we have $E_{a_n}\subset Q_{a_n}$, so that
$$P(E_{a_n})\le P(Q_{a_n})=2^d\bigg(\sum_{i=1}^d a_{in}^{-1}\bigg)\prod_{i=1}^d a_{in}\;.$$
Hence \eqref{ineqmax} implies, for a suitable constant $C_{q,d}$ depending only on $q$ and on $d$,
$$F_q(B_1)\le C_{q,d}\frac{\sum_{i=1}^d a_{in}^{-1}}{\big(\sum_{i=1}^d a_{in}^{-2}\big)^q\big(\prod_{i=1}^d a_{in}\big)^{(2q-1)/d}}
\le d^qC_{q,d}\bigg(\frac{\big(\prod_{i=1}^d a_{in}^{-1}\big)^{1/d}}{\sum_{i=1}^d a_{in}^{-1}}\bigg)^{2q-1},$$
where in the last inequality we used the Cauchy-Schwarz inequality \eqref{stimamid_a1}. Finally, since $a_{in}\le a_{dn}=1$, we obtain
$$F_q(B_1)\le d^q C_{q,d}(a_{in}^{-1})^{(2q-1)/d}$$
and, since $q>1/2$, the conclusion follows as in the previous case.
\end{proof}

%%%%%%%%%%%%%%%%%%%%%%%%%%%%%%%%%%%%%%%%%%%%%%%%%%
\section{Optimization in the class of thin domains\label{sthin}}

In this section we consider the class of {\it thin domains}
$$\O_\eps=\big\{(s,t)\ :\ s\in A,\ \eps h_-(s)<t<\eps h_+(s)\big\}$$
where $\eps$ is a small positive parameter, $A$ is a (smooth) domain of $\R^{d-1}$, and $h_-,h_+$ are two given (smooth) functions. We denote by $h(s)$ the {\it local thickness}
\[h(s)=h_+(s)-h_-(s),\]
and we assume that $h(s)\ge0$. The following asymptotics hold for the quantities we are interested to (for the torsional rigidity we refer to \cite{bofr13}):
\[\begin{split}
&P(\O_\eps)\approx2\HH^{d-1}(A),\\
&T(\O_\eps)\approx\frac{\eps^3}{12}\int_A h^3(s)\,ds,\\
&|\O_\eps|=\eps\int_A h(s)\,ds,
\end{split}\]
which together give the asymptotic formula when $q=1/2$
\be\label{asympt}
\begin{split}
F_{1/2}(\O_\eps)&\approx 3^{-1/2}\HH^{d-1}(A)\Big[\int_A h^3(s)\,ds\Big]^{1/2}\Big[\int_A h(s)\,ds\Big]^{-3/2}\\
&=3^{-1/2}\bigg[\Big[\avint_{\hskip-5pt A}h^3(s)\,ds\Big]\Big[\avint_{\hskip-5pt A}h(s)\,ds\Big]^{-3}\bigg]^{1/2}
\end{split}
\ee
where we use the notation
$$\avint_{\hskip-5pt A}f(s)\,ds=\frac{1}{\HH^{d-1}(A)}\int_A f(s)\,ds.$$
By H\"older inequality we have
$$\lim_{\eps\to0}F_{1/2}(\O_\eps)\ge3^{-1/2}$$
and the value $3^{-1/2}$ is actually reached by taking the local thickness function $h$ constant, which corresponds to $\O_\eps$ a thin {\it slab}.

A sharp inequality from above is also possible for $F_{1/2}(\O_\eps)$, if we restrict the analysis to {\it convex} domains, that is to local thickness functions $h$ which are {\it concave}. The following result will be used, for which we refer to \cite{bor73}, \cite{gar98}.

\begin{theo}\label{borell}
Let $1\le p\le q$. Then for every convex set $A$ of $\R^N$ $(N\ge1)$ and every nonnegative concave function $f$ on $A$ we have
$$\Big[\avint_{\hskip-5pt A}f^q\,dx\Big]^{1/q}\le C_{p,q}\Big[\avint_{\hskip-5pt A}f^p\,dx\Big]^{1/p}$$
where the constant $C_{p,q}$ is given by
$$C_{p,q}=\binom{N+p}{N}^{1/p}\binom{N+q}{N}^{-1/q}.$$
In addition, the inequality above becomes an equality when $A$ is a ball of radius $1$ and $f(x)=1-|x|$.
\end{theo}

We are now in a position to prove the Conjecture \ref{conjecture} for convex thin domains.

\begin{theo}
If $\O_\eps$ are thin convex domains with local thickness $h$, we have
\be\label{equalthin}
\lim_{\eps\to0}F_{1/2}(\O_\eps)\le d\Big(\frac{2}{(d+1)(d+2)}\Big)^{1/2}.
\ee
In addition, the inequality above becomes an equality taking for instance as $A$ the unit ball of $\R^{d-1}$ and as the local thickness $h(s)$ the function $1-|s|$.
\end{theo}

\begin{proof}
By \eqref{asympt} we have
$$\lim_{\eps\to0}F_{1/2}(\O_\eps)=3^{-1/2}\bigg[\Big[\avint_{\hskip-5pt A}h^3(s)\,ds\Big]\Big[\avint_{\hskip-5pt A}h(s)\,ds\Big]^{-3}\bigg]^{1/2}.$$
In addition, by Theorem \ref{borell} with $N=d-1$, $q=3$, $p=1$, we obtain
$$\avint_{\hskip-5pt A}h^3\,dx\le C_{1,3}^3\Big[\avint_{\hskip-5pt A}h\,dx\Big]^3\;,$$
so that
$$\lim_{\eps\to0}F_{1/2}(\O_\eps)\le3^{-1/2}C_{1,3}^{3/2}=d\Big(\frac{2}{(d+1)(d+2)}\Big)^{1/2}$$
as required. Finally, an easy computation shows that in \eqref{equalthin} the inequality becomes an equality if $A$ is the unit ball of $\R^{d-1}$ and $h(s)=1-|s|$.
\end{proof}

\bigskip

\noindent{\bf Acknowledgments.} The work of GB is part of the project 2017TEXA3H {\it``Gradient flows, Optimal Transport and Metric Measure Structures''} funded by the Italian Ministry of Research and University. The authors are member of the Gruppo Nazionale per l'Analisi Matematica, la Probabilit\`a e le loro Applicazioni (GNAMPA) of the Istituto Nazionale di Alta Matematica (INdAM).

\bigskip
{\small\noindent
Luca Briani:
Dipartimento di Matematica,
Universit\`a di Pisa\\
Largo B. Pontecorvo 5,
56127 Pisa - ITALY\\
{\tt luca.briani@phd.unipi.it}

\bigskip\noindent
Giuseppe Buttazzo:
Dipartimento di Matematica,
Universit\`a di Pisa\\
Largo B. Pontecorvo 5,
56127 Pisa - ITALY\\
{\tt giuseppe.buttazzo@dm.unipi.it}\\
{\tt http://www.dm.unipi.it/pages/buttazzo/}

\bigskip\noindent
Francesca Prinari:
Dipartimento di Matematica e Informatica,
Universit\`a di Ferrara\\
Via Machiavelli 30,
44121 Ferrara - ITALY\\
{\tt francescaagnese.prinari@unife.it}\\
{\tt http://docente.unife.it/francescaagnese.prinari/}}

\end{document}